\documentclass[11pt,reqno]{amsart}

\usepackage[utf8]{inputenc}

\usepackage{amsmath,amsthm,amssymb}

\usepackage{graphics}
\usepackage{hyperref}
\usepackage[usenames, dvipsnames]{xcolor}

\usepackage{mathtools}
\mathtoolsset{showonlyrefs}

\usepackage[square,sort,comma,numbers]{natbib}

\definecolor{darkblue}{rgb}{0.0,0.0,0.3}
\hypersetup{colorlinks,breaklinks,
  linkcolor=darkblue,urlcolor=darkblue,
anchorcolor=darkblue,citecolor=darkblue}

\theoremstyle{plain}
\newtheorem{theorem}{Theorem}
\newtheorem*{theorem*}{Theorem}

\newtheorem*{proposition*}{Proposition}

\newtheorem*{corollary*}{Corollary}

\theoremstyle{definition}

\renewcommand{\Re}{\operatorname{Re}}
\DeclareMathOperator{\SL}{SL}

\DeclareMathOperator{\MT}{MT}

\title{Non-real Poles and Irregularity of Distribution}
\author{David Lowry-Duda}
\thanks{This work was supported by the Simons Collaboration in Arithmetic
Geometry, Number Theory, and Computation via the Simons Foundation grant 546235.}
\thanks{The author wants to thank Thomas Hulse, Henryk Iwaniec, Chan Ieong Kuan,
and Alexander Walker for their encouraging remarks.}

\begin{document}

\maketitle

\begin{abstract}

We study the general theory of weighted Dirichlet series and associated
summatory functions of their coefficients. We show that any non-real pole leads
to oscillatory error terms. This applies even if there are infinitely many
non-real poles with the same real part. Further, we consider the case when the
non-real poles lie near, but not on, a line. The method of proof is a
generalization of classical ideas applied to study the oscillatory behavior of
the error term in the prime number theorem.

\end{abstract}

\section{Overview and Results}

In this article, we study the general theory of Dirichlet series
$D(s) = \sum_{n \geq 1} a(n) n^{-s}$ and the associated summatory function of
the coefficients, $A(x) = \sum_{n \leq x}' a(n)$ (where the prime over the
summation means the last term is to be multiplied by $1/2$ if $x$ is an
integer). For convenience, we will suppose that the coefficients $a(n)$ are
real, that not all $a(n)$ are zero, that each Dirichlet series converges in some
half-plane, and that each Dirichlet series has meromorphic continuation to
$\mathbb{C}$. Perron's formula
shows that $D(s)/s$ and $A(x)$ are duals and satisfy
\begin{equation}\label{eq:basic_duality}
  \frac{D(s)}{s}
  =
  \int_1^\infty \frac{A(x)}{x^{s+1}} dx,
  \quad
  A(x)
  =
  \frac{1}{2 \pi i} \int_{\sigma - i \infty}^{\sigma + i \infty}
  \frac{D(s)}{s}
  x^s ds
\end{equation}
for an appropriate choice of $\sigma$.

Many results in analytic number theory take the form of showing that $A(x) =
\MT(x) + E(x)$ for a ``Main Term'' $\MT(x)$ and an ``Error Term'' $E(x)$.
Typically the terms in the main term $\MT(x)$ correspond to poles from
$D(s)$, while $E(x)$ is hard to understand. Upper bounds for the error term give
bounds for how much $A(x)$ can deviate from the expected size, and thus describe
the regularity in the distribution of the coefficients $\{a(n)\}$. In this
article, we investigate lower bounds for the error term, corresponding to
\emph{irregularity in the distribution} of the coefficients.

More generally, we can consider a sequence of denominators $0 < \lambda_1 <
\lambda_2 < \cdots$, where $\lambda_n \to \infty$, and consider the generalized
Dirichlet series
\begin{equation*}
  D_\lambda(s) = \sum_{\lambda_n} \frac{a(n)}{\lambda_n^s}
\end{equation*}
with the associated summatory function $A^{(\lambda)}(x) = \sum_{\lambda_n \leq
x}' a(n)$. We note that one can assume without loss of generality that
$\lambda_1 = 1$ by considering the related Dirichlet series
\begin{equation*}
  \lambda_1^s\sum_{\lambda_n} \frac{a(n)}{\lambda_n^s}
  =
  \sum_{\lambda_n} \frac{a(n)}{{(\lambda_n / \lambda_1)}^s}
\end{equation*}
and then translating results.
Thus we will assume our series have been normalized so that $\lambda_1 = 1$.

Results of the form $A^{(\lambda)}(x) = \MT(x) + E(x)$ for these
generalized Dirichlet series are also prevalent in number theory. For example,
one can associate an Epstein zeta function of the form $D_\lambda(s)$ to any
complete lattice $\Lambda \subset \mathbb{R}^d$; in this example, $A_\lambda$
counts the number of lattice points contained within a ball of a given radius.

Further, to get satisfactory understanding of the error terms, it is often
necessary to work with smoothed sums $A^{(\lambda)}_v(x) = \sum_{\lambda_n} a(n)
v(\lambda_n/x)$ for a weight function $v(\cdot)$. In this article, we consider
\emph{nice} weight functions, i.e.\ weight functions with good behavior and
whose Mellin transforms have good behavior. For almost all applications, it
suffices to consider weight functions $v(x)$ on the positive real numbers that
are piecewise smooth and that take values halfway between jump discontinuities.

For such a weight function $v(\cdot)$, denote its Mellin transform by
\begin{equation}
  V(s) = \int_0^\infty v(x)x^{s} \frac{dx}{x}.
\end{equation}
Then we can study the more general dual family
\begin{equation}\label{eq:general_duality}
  D_\lambda(s) V(s)
  =
  \int_1^\infty \frac{A^{(\lambda)}_v(x)}{x^{s+1}} dx,
  \quad
  A^{(\lambda)}_v(x)
  =
  \frac{1}{2 \pi i} \int_{\sigma - i \infty}^{\sigma + i \infty}
  D_\lambda(s) V(s)
  x^s ds.
\end{equation}

We prove two results governing the irregularity of distribution of weighted
sums. Our first theorem guarantees that a non-real pole of $D_\lambda(s)V(s)$
corresponds to an oscillatory error term for $A^{(\lambda)}_v(x)$.

\begin{theorem}\label{thm:single_pole}
  Suppose $D_\lambda(s)V(s)$ has a pole at $s = \sigma_0 + it_0$ with $t_0 \neq
  0$ of order $r$. Let $\MT(x)$ be the sum of the residues of
  $D_\lambda(s)V(s)X^s$ at all real poles $s = \sigma$ with $\sigma \geq
  \sigma_0$.

  Then
  \begin{equation}
    \sum_{\lambda_n} a(n) v(\tfrac{\lambda_n}{x}) - \MT(x)
    =
    \Omega_\pm\big( x^{\sigma_0} \log^{r-1} x\big).
  \end{equation}
\end{theorem}

Here and below, we use the notation $f(x) = \Omega_+(g(x))$ to mean that there
is a constant $k > 0$ such that $\limsup f(x)/\lvert g(x) \rvert > k$ and $f(x)
= \Omega_- (g(x))$ to mean that $\liminf f(x)/\lvert g(x) \rvert < -k$. When
both are true, we write $f(x) = \Omega_\pm (g(x))$. This roughly means that
$f(x)$ is about as positive as $\lvert g(x) \rvert$ and about as negative as
$-\lvert g(x) \rvert$ infinitely often.

Our second theorem concerns the case when there are infinitely many non-real
poles very near a line.

\begin{theorem}\label{thm:many_poles}
  Suppose $D_\lambda(s)V(s)$ has at least one non-real pole, and that the
  supremum of the real parts of the non-real poles of $D_\lambda(s)V(s)$ is
  $\sigma_0$. Let $\MT(x)$ be the sum of the residues of $D_\lambda(s)V(s)X^s$
  at all real poles $s = \sigma$ with $\sigma \geq \sigma_0$.

  Then for any $\epsilon > 0$,
  \begin{equation}
    \sum_{\lambda_n} a(n) v(\tfrac{\lambda_n}{x}) - \MT(x)
    =
    \Omega_\pm( x^{\sigma_0 - \epsilon} ).
  \end{equation}
\end{theorem}

The idea at the core of these theorems is old, and was first noticed during the
investigation of the error term in the prime number theorem. To prove them, we
generalize those techniques. In particular, we generalize the proofs given in
Chapter 5 of Ingham's 1932 monograph~\cite{ingham1990distribution} (recently
republished). In his monograph, Ingham describes a proof showing that $\psi(x) -
x = \Omega_\pm(x^{\Theta - \epsilon})$ and $\psi(x) - x = \Omega_\pm(x^{1/2})$,
where $\psi(x) = \sum_{p^n \leq x} \log p$ is Chebyshev's second function and
$\Theta \geq \frac{1}{2}$ is the supremum of the real parts of the non-trivial
zeros of $\zeta(s)$. We note that modern monographs or summaries on topics
concerning analysis of $\zeta(s)$ will have these results, but as far as the
author knows they are never treated in the same level of generality as we
present here.

Thus in this article, we generalize these techniques and extend them to
weighted sums and more general Dirichlet series.

\section*{Motivation and Application}

Infinite lines of poorly understood poles appear regularly while studying
shifted convolution series of the shape
\begin{equation}
  D(s) = \sum_{n \geq 1} \frac{a(n) a(n \pm h)}{n^s}
\end{equation}
for a fixed $h$. When $a(n)$ denotes the (non-normalized) coefficients of a
weight $k$ cuspidal Hecke eigenform on a congruence subgroup of $\SL(2,
\mathbb{Z})$, for instance, one can provide meromorphic continuation for the
shifted convolution series $D(s)$ via spectral expansion in terms of Maass
forms and Eisenstein series. The Maass forms contribute infinite lines of
poles.

Explicit asymptotics would take the form
\begin{equation}
  \sum_{n \geq 1} a(n)a(n-h) e^{-n/X}
  =
  \sum_j C_j X^{\frac{1}{2} + \sigma_j + it_j} \log^m X
\end{equation}
where neither the residues nor the imaginary parts $it_j$ are well-understood.
Is it be possible for these infinitely many rapidly oscillating terms to
experience massive cancellation for all $X$? The theorems above prove that this
is not possible.

In this case, applying Theorem~\ref{thm:single_pole} with the Perron-weight
\begin{equation}
  v(x)
  =
  \begin{cases}
    1 & x < 1 \\
    \frac{1}{2} & x = 1 \\
    0 & x > 1
  \end{cases}
\end{equation} shows that
\begin{equation}
  \sideset{}{'}\sum_{n \leq X} \frac{a(n)a(n-h)}{n^{k-1}}
  =
  \Omega_\pm(\sqrt X).
\end{equation}
Similarly, Theorem~\ref{thm:many_poles} shows that
\begin{equation}
  \sideset{}{'}\sum_{n \leq X} \frac{a(n)a(n-h)}{n^{k-1}}
  =
  \Omega_\pm(X^{\frac{1}{2} + \Theta - \epsilon}),
\end{equation}
where $\Theta < 7/64$ is the supremum of the deviations to Selberg's Eigenvalue
Conjecture (sometimes called the non-arithmetic Ramanujan Conjecture).

More generally, these shifted convolution series appear when studying the sizes
of sums of coefficients of modular forms. A few years ago, Hulse, Kuan, Walker,
and the author began an investigation of the Dirichlet series whose coefficients
were themselves $\lvert A(n) \rvert^2$ (where $A(n)$ is the sum of the first $n$
coefficients of a modular form). We showed that this Dirichlet series has
meromorphic continuation to $\mathbb{C}$~\cite{hkldw1}, but also that there are
lines of poles coming from Maass forms as described above. The behavior of the
infinite lines of poles in the discrete spectrum play an important role in the
analysis, but we did not yet understand how they affected the resulting
asymptotics. The author intends to revisit these results, and others, from the
context of this article.

\section{Proofs}\label{sec:proofs}

Given a function $f(x)$ that is bounded and integrable over any finite interval,
the Dirichlet integral
\begin{equation}
  F(s) = \int_1^\infty \frac{f(x)}{x^{s+1}} dx
\end{equation}
behaves in many ways like a Dirichlet series. Most classical results for
Dirichlet series correspond to analogous results for Dirichlet integrals. For
instance, one can prove that if a Dirichlet integral converges at $s = \sigma_1
+ i t_1$, then it converges for all $s$ with $\Re s > \sigma_1$ in the same way
one proves it for Dirichlet series, except one replaces summation-by-parts with
integration-by-parts. Thus a Dirichlet integral will also have a well-defined
abscissa of convergence.

We will make repeated reference to the following theorem of Landau. It is often
presented only for Dirichlet series, but it also is true for Dirichlet
integrals.

\begin{theorem}[A Theorem of Landau]
  Suppose that $f(x)$ is bounded, integrable in any finite interval, and of
  constant sign for all sufficiently large $x$. Then the real part $s =
  \sigma_1$ of the abscissa of convergence of the Dirichlet integral $F(s)$ is a
  singularity.
\end{theorem}

\begin{proof}
  This is Theorem H in~\cite{ingham1990distribution}. The proof is
  analogous to proofs for the Dirichlet series version.
\end{proof}

Since $D_\lambda(s)$ is meromorphic on $\mathbb{C}$ and has a half-plane of
convergence, there are at most finitely many real poles in any half-plane $\Re s
> \sigma$.
Near a pole $s = u$ of order $r$, $D_\lambda(s) V(s)$ has a Laurent expansion
\begin{equation}\label{eq:DV_laurent}
  \sum_{m = 1}^{r} \frac{c_m}{{(s-u)}^m} + \phi_u(s)
\end{equation}
for an analytic function $\phi_u(s)$. Then the residue of $D_\lambda(s)V(s)X^s$
at $s = u$ can be written as
\begin{equation}\label{eq:DVX_residue}
  \sum_{m = 1}^{r} \frac{c_m}{(m-1)!}X^m \log^{m - 1} X.
\end{equation}
Finally, note that
\begin{equation}\label{eq:integral_eval}
  \int_1^\infty \frac{{(\log x)}^m}{x^{s + 1}} dx
  =
  \frac{m!}{s^{m+1}}.
\end{equation}

We are now ready to prove the two $\Omega_\pm$ theorems.

\begin{proof}[Proof of Theorem~\ref{thm:many_poles}]

Let $\sigma_0$ be the supremum of the real parts of the poles poles $s = \sigma
+ it$ with $t \neq 0$ of $D_\lambda(s)V(s)$; let $\MT(x)$ denote the sum of the
residues at the (finitely many) real poles $s = \sigma + it$ of
$D_\lambda(s)V(s)X^s$ with $\sigma \geq \sigma_0$.

Let $f(x) = A^{(\lambda)}_v(x) - \MT(x)$, and consider the Dirichlet integral
\begin{equation}\label{eq:def_Fs}
  F(s) = \int_1^\infty \frac{f(x)}{x^{s+1}} dx.
\end{equation}
For $s$ in the half-plane of absolute convergence of $D_\lambda(s)V(s)$, we
recognize
\begin{equation}
  F(s)
  =
  \int_1^\infty \frac{A^{(\lambda)}_v(x)}{x^{s+1}} dx
  -
  \int_1^\infty \frac{\MT(x)}{x^{s+1}} dx
  =
  D_\lambda(s)V(s) - (**)
\end{equation}
where $(**)$ is a sum of terms of the form~\eqref{eq:integral_eval}. Thus $F(s)$
has meromorphic continuation to $\mathbb{C}$ and the poles of $F(s)$ are a
subset of the poles of $D_\lambda(s)V(s)$. At any pole $s = \sigma + it$ with
$\sigma \geq \sigma_0$, one can check that the Laurent
expansion~\eqref{eq:DV_laurent} cancels with the part of $\int_1^\infty
\MT(x)x^{-s-1}dx$ corresponding to that pole. Thus $F(s)$ has no real pole $s =
\sigma + it$ with $\sigma \geq \sigma_0$.

Fix any real $k, c$ with $k > 0$ and $c < \sigma_0$. Define
\begin{equation}
  g(x) := A^{(\lambda)}_v(x) - \MT(x) - k x^c.
\end{equation}
Suppose for the sake of contradiction that $g(x) > 0$ for all sufficiently large
$x$. Then by Landau's Theorem, we know that the abscissa of convergence of
\begin{equation}
  G(s) = \int_1^\infty \frac{g(x)}{x^{s+1}}dx = F(s) - \frac{k}{(s - c)}
\end{equation}
is real. Denote this abscissa of convergence by $\sigma_1$. As $F(s)$ has a
non-real pole with real part greater than $\sigma_0 - \epsilon$ for any
$\epsilon > 0$, we have that $\sigma_1 \geq \sigma_0$.

On the other hand, any real pole $s = \sigma + it$ of $G(s)$ with $\sigma \geq
\sigma_0$ must come from $F(s)$. But we have demonstrated that $F(s)$ has no
real poles with real part at least $\sigma_0$. By contradiction, we have shown
that $A^{(\lambda)}_v(x) - \MT(x) = \Omega_-(x^c)$ for any $c < \sigma_0$. A similar
contradiction proves the corresponding $\Omega_+$ result, concluding the proof.
\end{proof}

The proof of Theorem~\ref{thm:single_pole} begins in the same way.

\begin{proof}[Proof of Theorem~\ref{thm:single_pole}]

Let $\sigma_0 + it_0, t_0 \neq 0$ denote a non-real pole of $D_\lambda(s)V(s)$
of order $r$, and let $\MT(x)$ denote the sum of residues at the real poles $s =
\sigma + it$ with $\sigma \geq \sigma_0$. As in the proof of
Theorem~\ref{thm:many_poles}, let $f(x) = A^{(\lambda)}_v(x) - \MT(x)$; consider
the Dirichlet integral $F(s)$ from~\eqref{eq:def_Fs}; it remains true that
$F(s)$ has no real poles with real part $\geq \sigma_0$.

Let $g_k(x)$ denote
\begin{equation}
  g_k(x) = A^{(\lambda)}_v(x) - \MT(x) - k x^{\sigma_0} \log^{r-1} x
\end{equation}
for a real constant $k > 0$ to be specified later. Let us suppose that $g_k(x) >
0$ for all sufficiently large $x$. We will show that taking $k$ sufficiently
small will lead to a contradiction.

By Landau's Theorem, we know that the abscissa of convergence, $\sigma_1$, of
\begin{equation}\label{eq:second_residue}
  G_k(s) = \int_1^\infty \frac{g_k(x)}{x^{s+1}} dx
  =
  F(s) - \frac{k/(r-1)!}{{(s-\sigma_0)}^r}
\end{equation}
is real. The pole at $\sigma_0 + it_0$ guarantees that $\sigma_1 \geq \sigma_0$
(but the pole may correspond to the pole coming from $x^\sigma \log^{r-1} x$).

Fix $Y$ such that $g_k(x) > 0$ for all $x > Y$. For any $\sigma > \sigma_1$, we
have that
\begin{equation}
  \lvert G_k(\sigma + it_0) \rvert
  \leq
  \int_1^Y \frac{\lvert g_k(x) \rvert}{x^{\sigma + 1}} dx
  +
  \int_Y^\infty \frac{g_k(x)}{x^{\sigma + 1}} dx
  =
  G_k(\sigma)
  +
  \int_1^Y
    \frac{\lvert g_k(x) \rvert - g_k(x)}{x^{\sigma + 1}}
  dx.
\end{equation}
Note also that
\begin{equation}
  \int_1^Y
    \frac{\lvert g_k(x) \rvert - g_k(x)}{x^{\sigma + 1}}
  dx
  \leq
  2 \int_1^Y \frac{\lvert g_k(x) \rvert}{x^{\sigma_0 + 1}} dx = K
\end{equation}
for a finite constant $K$ that is independent of $\sigma$ and $t$.
Thus
\begin{equation}\label{eq:pre_contradiction}
  \lvert G_k(\sigma + it_0) \rvert
  \leq
  K + G_k(\sigma).
\end{equation}

Multiply each side of~\eqref{eq:pre_contradiction} by ${(\sigma - \sigma_0)}^r$.
We will evaluate the limit as $\sigma \to \sigma_0$ from the right. In terms of
the Laurent expansions, this has the effect of comparing the (absolute value of
the) Laurent coefficients of ${(s - (\sigma_0 + it_0))}^{-r}$ in $G_k(\sigma +
it_0)$ on the left and of ${(s - \sigma_0)}^{-r}$ in $G_k(\sigma)$ on the right.
In particular, if we specialize~\eqref{eq:DV_laurent} to the pole at $\sigma_0 +
it_0$, then $\lim_{\sigma \to \sigma_0^+} \lvert G_k(\sigma + it_0) {(\sigma -
\sigma_0)}^r \rvert = \lvert c_r \rvert$. On the other hand,
from~\eqref{eq:second_residue}, we see that $\lim_{\sigma \to \sigma_0^+}(K +
G_k(\sigma)) = k/{(r-1)!}$.

Thus multiplying each side of~\eqref{eq:pre_contradiction} by ${(\sigma -
\sigma_0)}^r$ and taking the limit as $\sigma \to \sigma_0$ from the right gives
the inequality
\begin{equation}
  \lvert c_r \rvert \leq \frac{k}{(r-1)!}.
\end{equation}
But $k$ can be chosen freely, and it is clear that choosing $k < (r-1)!\lvert
c_r \rvert$ leads to a contradiction. Thus $g_k(x) = \Omega_-(x^{\sigma_0}
\log^{r-1} x)$. The $\Omega_+$ result is proved similarly.
\end{proof}

\vspace{20 mm}
\bibliographystyle{alpha}
\bibliography{bibfile}

\end{document}